\documentclass[11pt]{article}
\usepackage{amsmath, amsthm, amssymb} 
\usepackage{hyperref}

\newcommand{\be}{\begin{equation}}
\newcommand{\ee}{\end{equation}}
\newcommand{\bea}{\begin{eqnarray}}
\newcommand{\eea}{\end{eqnarray}}
\newcommand{\beas}{\begin{eqnarray*}}
\newcommand{\eeas}{\end{eqnarray*}}

\newcommand{\bfG}{{\bf \Gamma}}

\newcommand{\bbE}{\mathbb E}

\newcommand{\bbN}{\mathbb N}

\newcommand{\bbR}{\mathbb R}

\newcommand{\scE}{\mathcal E}
\newcommand{\scF}{\mathcal F}
\newcommand{\scG}{\mathcal G}

\newcommand{\scP}{\mathcal P}

\newcommand{\norm}[1]{\ensuremath{\left\| #1 \right\|}}
\newcommand{\abs}[1]{\ensuremath{\left| #1 \right|}}

\newcommand{\indicator}[1]{\ensuremath{\mathbf{1}_{\crl{#1}}}}
\newcommand{\Ito}{It\ensuremath{\hat{\textrm{o}}}}
\newcommand{\half}{\frac{1}{2}}

\newcommand{\crl}[1]{\ensuremath{ \left\{ #1 \right\} }}

\newcommand{\brak}[1]{\ensuremath{\left( #1 \right)}}

\newtheorem{theorem}{Theorem}[section]
\newtheorem{definition}[theorem]{Definition}
\newtheorem{proposition}[theorem]{Proposition}
\newtheorem{corollary}[theorem]{Corollary}
\newtheorem{lemma}[theorem]{Lemma}
\newtheorem{remark}[theorem]{Remark}
\newtheorem{example}[theorem]{Example}
\newtheorem{examples}[theorem]{Examples}
\newtheorem{foo}[theorem]{Remarks}
\newenvironment{Example}{\begin{example}\rm}{\end{example}}

\newenvironment{Remark}{\begin{remark}\rm}{\end{remark}}

\title{Variations of dynamic random networks:\\localization approach\footnote{I would like to thank Harry Crane for introducing me a new field of research and fruitful discussions.}}
\author{
Kihun Nam\\
Monash University\\
Clayton, VIC 3800, Australia
}

\begin{document}
\maketitle
\begin{abstract}
We study the variation of exchangeable graph-valued process ${\bf \Gamma}$ and its graph limit. We used a constructive method using localization technique. Our method provides a specific estimation of variation for exchangeable graph-valued process ${\bf \Gamma}$ and its graph limit for different types of metric. As a result, we extend the bounded variation of graph limit as in Crane (2016, \cite{Crane:2015we}).
\\[2mm]
\end{abstract}
\setcounter{equation}{0}
\section{Introduction}\label{intro}
\label{sec:intro}

Time-varying networks are everywhere: neuronal networks in human brain, social networks in Facebook or LinkedIn, systemic risk models in finance, virus spreading models in epidmiology, terrorist network model for national security, and many others. In his recent paper, Crane \cite{Crane:2016vw,Crane:2015ef,Crane:2015we}) studied fundamental properties of exchangeable c\`adl\`ag graph-valued processes such as Levy characterization for Markovian case and variation estimates for graph limits for general case. In particular, Crane (\cite{Crane:2015we}) showed that the graph limit of $\bfG$ necessarily have (finite-dimensional) bounded variation under certain
metric on the space of graph limit. This implies that the graph-limit process of exchangeable c\`adl\`ag dynamic random networks cannot have local martingale term in their Doob-Meyer decomposition, and in turns, it means that we do not need to consider second-order term when we apply {\Ito} formula for the graph-limit process. 

However, there are three limitations in his result. First of all, the result in \cite{Crane:2015we} depends on the metric one imposes on the space of graph limit. Since there is no natural metric on the space of graph limit, a universal method for the estimation of variation for different metrics would be highly desirable. Second, it only provides the result of variation on graph limit not the graph-process itself. Lastly, as his proof is not constructive, the paper does not provide a specific estimation for the variations.

Our aim in this article is to provide satisfactory answers to above limitations using constructive approach. We use time-localization technique to estimate variations of exchangeable graph-valued processes and their graph-limit process. Our technique enables us to provide estimate for different types of metrics. Moreover, we are able to directly estimate the variation of graph-valued process itself, rather than studying its graph limit. Additionally, the method enables us to improve Theorem 3.3 of \cite{Crane:2015we}: we were able to define appropriate metric on the space of graph limit so that infinite-dimensional total variation is also bounded as well as finite-dimensional total variation.

\section{Preliminaries}
A graph is a set of vertices $V$ with its connecting edges. Throughout the
paper, all graphs are undirected and have no self-loop. We denote the set of
graphs with $n$ vertices as $\scG_{n}$ and let
$V=[n]:=\crl{1,2,...,n}$ be the vertex set. When $V$ is countably many
vertices, then we set $V=\bbN:=\crl{1,2,...}$ and denote such infinite
graph as $\scG_{\infty}$. We denote
$\bar\scG:=\cup_{n\in\bbN}\scG_{n}$. 

 Since we are concerning about undirected graph $G$, it can be
represented by its adjacency matrix $(G^{ij})_{i,j\in V}$ for which
\begin{align*}
G^{ij}=\left\{\begin{array}{ll}1&\text{ if vertices $i$ and $j$ are connected by an edge}\\
  0&\text{ otherwise}\end{array}\right..
\end{align*}
Note that $(G^{ij})_{i,j\in V}$ is symmetric and have zero diagonal entries because we
assumed it
is undirected and has no self-loop. 

We define a function $J_{n}:\scG_{n}\times\scG_{n}\to\bbN$
\[
J_{n}(F,G)=\frac{{\rm tr}(F-G)(F-G)^{T}}{n(n-1)}=\frac{\sum_{i,j\in[n]}(F^{ij}-G^{ij})^{2}}{n(n-1)},
\]
where we treat $F$ and $G$ as adjacency matrices. The $J_n$ represents for the portion of differences between $F$ and $G$ and it can be seemed as a graph edit distance. Note that $J_n(F,G)$ depends on the labeling of vertices.  For $F, G\in \scG_{\infty}$, we define, if it exists,
\[
J(F,G):=\lim_{n\to\infty}J_{n}(F|_n, G|_n)
\]
where $F|_n$ and $G|_n$ is the restriction of $F$ and $G$ for the first $n$ vertices. Note that $J$ is not a metric on $\scG_\infty$ because $J(F,G)=0$ does not imply $F=G$.

For a permuation
$\sigma:\bbN\to\bbN$, we denote
\begin{align*}
G^{\sigma}=(G^{\sigma(i)\sigma(j)})_{i,j\in\bbN}.
\end{align*}
In general, for an injective function $\varphi:[n]\to[m]$, we denote
\begin{align*}
G^{\varphi}=(G^{\varphi(i)\varphi(j)})_{i,j\in[n]} .
\end{align*}


In this article, 
for a graph-valued process $\bfG=(\Gamma_{t})_{t\geq 0}$ on
$\scG_{\infty}$, we assume the following conditions.
\begin{itemize}
\item[(i)](c\`adl\`ag) $\bfG|_n$ has c\`adl\`ag sample paths under product-discrete topology.
\item[(ii)](exchangeable) For any permutation $\sigma:\bbN\to\bbN$,
  $\bfG^{\sigma}$ is a version of $\bfG$.
\item[(iii)] (left-continuity under $J$) For any $t>0$,
$
J(\Gamma_{t},\Gamma_{t-})=0.
$
\item[(iv)] For all $0\leq s<t\leq 1$,  there is $u\in(s,t)$ such that 
$ J(\Gamma_u,\Gamma_s)>0.$
\end{itemize}
\begin{Remark}
Note that condition (i) and (ii) implies the well-definedness of $J(\Gamma_s,\Gamma_t)$ for any $s,t\in\bbR_+$ by the following Strong Law of Large Number exchangeable symmetric random array.
\begin{theorem}(SLLN, \cite{Eagleson:2008ix}) Let $X$ be an exchangeable symmetric random array with zero diagonal entries and $X^{ij}=1$ or $0$. Then,
	\[
	T_n:=\frac{2}{n(n-1)}\sum_{1\leq i<j\leq n}X^{ij}\xrightarrow[a.s.]{n\to\infty}\bbE[X^{12}|\cap_{n=1}^\infty\scF_n]
	\]
	where $\scF_n=\sigma(\crl{T_n,T_{n+1},...})$.
\end{theorem}
\end{Remark}
\begin{Remark}
The condition (iii) means that the discontinuities of $\bfG$ are only of type
($B$) in \cite{Crane:2015we}. In other words, when $\bfG$ changes, zero
portion of total edges changes. Note that we are dealing
with infinite graphs and this type of jumps includes any finite
simultaneous edge jumps. This condition can be relaxed easily. Note that from Corollary \ref{estN}, the event of nonzero portion of edges jumps simultaneously can happen only finite number of times. This term add only a finite number to the multi-order variation. Therefore, it does not affect whether the variations are finite or infinite.
\end{Remark}
\section{Variation of graph}
The c\`adl\`ag path of $\bfG=(\Gamma_t)_{t=1}^\infty$ implies the following lemma.
\begin{lemma}\label{billing}
	For any $i,j\in\bbN$, $\bfG^{ij}$ has finitely many jumps on $[0,1]$
	with probability $1$. 
\end{lemma}
\begin{proof}
	Note that $\bfG^{ij}$ is a c\`adl\`ag with value $1$ or $0$. By
	Billingsley (2009, p 122), one can select a sequence of points $t_{0},t_{1},\cdots, t_{v}$ such that
	$0=t_{0}<t_{1}<\cdots<t_{v}=1$ and 
	\begin{align*}
	\sup_{s,t\in[t_{k-1},t_{k})}|\Gamma^{ij}_{s}-\Gamma^{ij}_{t}|&<\half,\qquad k=1,2,...,v
	\end{align*}
	Therefore, jump can happen only at $t_{1},t_{2},...,
	t_{v}$. Therefore, $\bfG^{ij}$ has finitely many jumps on $[0,1]$.
\end{proof}
Let us define an increasing sequence of stopping times $(\tau^{p}_{k})_{k=0}^\infty$. We will show that this will be a localizing sequence. We define $\tau^{p}_{0}:=0$  and, for $k\geq 1$,
\[
\tau^{p}_{k}:=\inf\crl{t>\tau_{k-1}^{p}: J(\Gamma_{t},\Gamma_{\tau_{k-1}^{p}})=p}.
\]

Let us define
\[
N_{p}:=\inf\crl{k\in\bbN:\tau^{p}_{k}> 1}.
\]
Note that $\tau^{p}_{N_{p}-1}\leq 1<\tau^{p}_{N_{p}}$ if $N_p<\infty$.
Let $j(e)$ to be the number of jumps for the edge $e$ of $\bfG$ during $[0,1]$.
\begin{proposition}\label{jumpest}
	For almost every $\omega\in\Omega$, there exists an edge $e$, which may depend on $\omega$, such that 
	\[
	j(e)\geq \sup_{p\in(0,1)}pN_p(\omega)-1
	\]
\end{proposition}

\begin{proof}
	Since we have countable number of edges, let us denote with index
	$l=1,2,...$. Let us define random variables
	\begin{align*}
	X^{kl}=|\Gamma^{l}_{\tau^{p}_{k}}-\Gamma^{l}_{\tau^{p}_{k-1}}|.
	\end{align*}
	Note that the law of $\crl{X^{kl}}$ is invariant under relabeling of vertices because $(\tau^p_k)$ is not affected by relabeling. Therefore, it should be an exchangeable random array for given $k$.
Assume that with strictly positive probability, $j(l)<\sup_{p\in(0,1)}pN_p-1$ for all edge $l$. Note that $\sup_{p\in(0,1)}\sum_{k=1}^{N_{p}-1}X^{kl}$ is a lower bound for $j(l)$. Then we have
\begin{align*}
\sup_{p\in(0,1)}pN_p-1>\liminf_{n\to\infty}\frac{1}{n}\sum_{l=1}^nj(l)\geq\liminf_{n\to\infty}\frac{1}{n}\sum_{l=1}^n\sup_{p\in(0,1)}\brak{\sum_{k=1}^{N_{p}-1}X^{kl}}\geq\liminf_{n\to\infty}\sup_{p\in(0,1)}\brak{\frac{1}{n}\sum_{l=1}^n\sum_{k=1}^{N_{p}-1}X^{kl}}
\end{align*}
with strictly positive probability. Note that
\[
\liminf_{n\to\infty}\sup_{p\in(0,1)}\brak{\frac{1}{n}\sum_{l=1}^n\sum_{k=1}^{N_{p}-1}X^{kl}}\geq \sup_{p\in(0,1)}\brak{\liminf_{n\to\infty}\frac{1}{n}\sum_{l=1}^n\sum_{k=1}^{N_{p}-1}X^{kl}}.
\]
By Fubini theorem, Fatou's lemma, and the law of large number, we get
\[
\liminf_{n\to\infty}\frac{1}{n}\sum_{l=1}^{n}\sum_{k=1}^{N_{p}-1}X^{kl}\geq\sum_{k=1}^{N_{p}-1}\liminf_{n\to\infty}\frac{1}{n}\sum_{l=1}^{n}X^{kl}=pN_p-p.
\]
This implies
\[
\sup_{p\in(0,1)}pN_p-1>\sup_{p\in(0,1)}\brak{pN_p-p}.
\]
with strictly positive probability and this is a contradiction if $\sup_{p\in(0,1)}pN_p<\infty$. If $\sup_{p\in(0,1)}pN_p=\infty$, then we have $j(l)=\infty$ and it contradicts to Lemma \ref{billing}.
\end{proof}
\begin{corollary}\label{estN} For any $p\in(0,1)$,
	$N_p<\infty$ and $\sup_{p\in(0,1)}pN_p<\infty$ almost surely. Moreover, if there exists $p'\in(0,1)$ such that $N_{p'}>1$, then $\limsup_{p\to 0}pN_p>0$
\end{corollary}
\begin{proof}
	It is obvious from the previous proposition that $N_p$ and $\sup_{p\in(0,1)}pN_p$ should be finite almost surely. Otherwise,there exists an edge that jumps infinitely many times with positive probability, which contradicts Lemma \ref{billing}. Therefore, the first part of the claim is proved.
	
	On the other hand, note that $1\geq \tau^{p'}_{N_{p'}-1}\geq \tau^{p'/2}_{2(N_{p'}-1)}$ and therefore, $N_{p'/2}-1\geq 2N_{p'}-2$. This implies 
	\[
	\frac{p'}{2} (N_{p'/2}-1)\geq p'(N_{p'}-1).
	\]
	 and therefore, $(a_k)_{k=0}^\infty$ where $a_k:=p'2^{-k}(N_{p'2^{-k}}-1)$ is an increasing sequence bounded above by $\sup_{p\in(0,1)}pN_p$. As a result, $(a_k)_{k=0}^\infty$ converges and
	\[
	\limsup_{p\to 0}pN_p\geq\lim_{k\to\infty}p'2^{-k}(N_{p'2^{-k}}-1)\geq p'(N_{p'}-1)>0
	\]
	since $N_{p'}>1$.
\end{proof}
The following corollary says that $(\tau_k^p)$ is indeed a localizing sequence.
\begin{corollary}
	For any $p\in(0,1)$, $\tau^{p}_{i}\nearrow\infty$ almost surely.
\end{corollary}
\begin{proof}
	Assume otherwise: there exists $T\in\bbR$ such that
	$\tau^{p}_{k}\leq T$ for all $i\in\bbN$ with positive probability. Without loss of generality, we can assume $T=1$. This implies that $N_{p}=\infty$ with strictly positive probability which is a contradiction.
\end{proof}

Let us define the following $\alpha$-order variation of graph valued process.
\begin{definition}
	Let $d$ be a metric on $\scG_\infty$. Let $\alpha\in[1,\infty)$. For a graph valued process $\bfG=(\Gamma_t)_{t\geq 0}$ satisfying (i)--(iv), we define $\alpha$-order variation on $[0,1]$ by
	\[
	\norm{\bfG}_{\alpha[0,1]}:=\lim_{M\to\infty}\lim_{p\to0}\sum_{k=1}^{N_{p}}\abs{d\brak{P_M(\Gamma_{\tau^{p}_k}),P_M(\Gamma_{\tau^p_{k-1}})}}^\alpha
	\]
	where $P_m:\scG_\infty\to\scG_\infty$ is given by
	$(P_m(X))^{ij}:=X^{ij}\indicator{i,j\in[m]}$ for all $i,j\in\bbN$.
\end{definition}

The conventional definition of $\alpha$-order variation would be
\[
\lim_{\norm{\scP_N}\to0}\sum_{k=1}^N\abs{d(\Gamma_{t_i},\Gamma_{t_{i-1}})}^\alpha 
\]
where $\scP_N$ is a partition of $[0,1]$ with points $\crl{t_i:0=t_0<t_1<\cdots<t_N=1}$ and $\norm{\scP_N}:=\max_{i}|t_i-t_{i-1}|$. 
Our definition may seem slightly different in two aspects. First, we partition the time with localizing sequence of stopping times.
Indeed, if only $0$ portion of total edges of $\bfG$ jumps during $[0,1]$, this definition does not coinside with classical definition since $\tau^p_1=\infty$ for any $p>0$. However, under our assumption (iv), our definition is consistent with the conventional definition because $\tau^p_k-\tau^p_{k-1}\to0$ as $p\to0$. Another difference is that, in our definition, we localize $\bfG$ to the first $M$ vertices, calculate $\alpha$-variance, and then send $M$ to infinity. When $\bfG$ has finite number of vertices, it is consistent with the usual $\alpha$-order variation. Our motivation of the infinite graph comes from our attempt to approximate a large graph. In this sense, our definition is justified.

\begin{theorem}\label{graphvar}
	Let us denote perm$[m]$ to be the permutation on the first $m$ number of vertices. 
		Assume that, for $G,F\in\scG_\infty$,
		\[
		d(P_M(G),P_M(F))\nearrow d(G,F) \text{ as } M\nearrow \infty
		\]
	\[
	\lim_{m\to\infty}\frac{1}{m!}\sum_{\sigma\in\text{perm}[m]}|d(F^\sigma, G^\sigma)|^\alpha\leq CJ(F,G)
	\]
	for some $C$. Then, $\bfG$ has finite $\alpha$-variation almost surely.
\end{theorem}
\begin{proof} 
	For a function $X: perm(\bbN)\to \bbR$, define
	\[
	\scE(X):=\lim_{m\to\infty}\frac{1}{m!}\sum_{\sigma\in\text{perm}[m]}X(\sigma)
	\]
	Note that $\scE(d(F^\cdot,G^\cdot))$ is an expectation $d(F,G)$ under random indexing. For a given realization of $\Gamma(\omega)$, by Fatou lemma, we have
\begin{align}
	\scE(\norm{\bfG(\omega)}_{\alpha[0,1]})&=\scE\brak{\lim_{M\to\infty}\lim_{p\to0}\sum_{k=1}^{N_{p}(\omega)}\abs{d\brak{P_M\brak{\Gamma^\sigma_{\tau^{p}_k}(\omega)},P_M\brak{\Gamma^\sigma_{\tau^p_{k-1}}(\omega)}}}^\alpha}\\
	&\leq \liminf_{M\to\infty}\liminf_{p\to0}\sum_{k=1}^{N_{p}(\omega)}\scE\brak{\abs{d\brak{P_M\brak{\Gamma^\sigma_{\tau^{p}_k}(\omega)},P_M\brak{\Gamma^\sigma_{\tau^p_{k-1}}(\omega)}}}^\alpha}\\
	&\leq \liminf_{p\to0}\sum_{k=1}^{N_{p}(\omega)}\scE\brak{\abs{d\brak{\Gamma^\sigma_{\tau^{p}_k}(\omega),\Gamma^\sigma_{\tau^p_{k-1}}(\omega)}}^\alpha}\\
	&\leq C \liminf_{p\to0}pN_p(\omega)<\infty
\end{align}
for almost every $\omega$ by Corollary \ref{estN}. This implies that $\alpha$-order variation is finite almost surely.
\end{proof}
\begin{Example}
	On $[0,1]$, $\bfG$ satisfying (i)--(iv) has finite $\alpha$-order variation for all $\alpha>2$ under the metric
	\[
	d(F,G):=\frac{1}{\max\crl{n:F|_n=G|_n}}
	\]
	on $\scG_{\infty}$. 
\end{Example}
\begin{proof}
	The first condition of Theorem \ref{graphvar} is trivially satisfied. For $F,G\in\scG_\infty$, let $\brak{F-G}^{ij}=\indicator{F^{ij}\neq G^{ij}}$ and $p=J(F,G)=J(F-G,0)$.
	\begin{align*}
	\lim_{m\to\infty}\frac{1}{m!}\sum_{\sigma\in\text{perm}[m]}|d(F^\sigma, G^\sigma)|^\alpha &=\lim_{m\to\infty}\frac{1}{m!}\sum_{\sigma\in\text{perm}[m]}|d((F-G)^\sigma, 0)|^\alpha\\
	&= \sum_{n=1}^\infty\frac{1}{n^\alpha}(1-p)^{n\choose 2} (1-(1-p)^n)\\
	&\leq p\sum_{n=1}^\infty\frac{1}{n^{\alpha-1}}(1-p)^{n\choose 2}\\
	&\leq p\sum_{n=1}^\infty\frac{1}{n^{\alpha-1}}
	\end{align*}
	where we used convention ${1 \choose 2}=0$. Therefore, if $\alpha>2$, then $\bfG$ has finite $\alpha$-order variation by our previous theorem.
\end{proof}
\section{Variation of graph limit}
\begin{definition}
	The density of $F\in\scG_{n}$ in $G\in\scG_{\infty}$ is defined by
	\[
	t(F;G):=\lim_{m\to\infty}\frac{1}{m^{n\downarrow}}\sum_{\varphi:[n]\to[m],
		injective}\indicator{G^{\varphi}=F}
	\]
	where $m^{n\downarrow}=m(m-1)\cdots(m-n+1)$.   
\end{definition}
The graph limit is an ordered set of graph densities.
\begin{definition}
	The graph limit of $G\in\scG_{\infty}$ is
	\[
	|G|:=(t(F;G))_{F\in\bar\scG}.
	\]
	We denote $|\bfG|:=(|\Gamma_{t}|)_{t\geq 0}$.
\end{definition}
The existence of graph limit is proved in \cite{Crane:2015we} and we will take it for granted.
\begin{proposition}
	Let $F\in\scG_n$ and $G,H\in\scG_\infty$ and assume that $J(F,G)$ exists. Then, 
	\begin{align*}
	\abs{t(F;G)-t(F;H)}\leq{n\choose 2} J(G,H)
	\end{align*}
\end{proposition}
\begin{proof}
	Let us denote $p:=J(G,H)$. Note that
	\begin{align*}
	|t(F;G)-t(F;H)|&=\abs{\lim_{m\to\infty}\frac{1}{m^{n\downarrow}}\sum_{\varphi:[n]\to[m],\text{\rm injective}}\brak{\indicator{G^\varphi=F}-\indicator{H^\varphi=F}}}\\
	&\leq\lim_{m\to\infty}\frac{1}{m^{n\downarrow}}\sum_{\varphi:[n]\to[m],\text{\rm injective}}\indicator{G^\varphi\neq H^\varphi}\\
	&\leq 1-\lim_{m\to\infty}\frac{1}{m^{n\downarrow}}\sum_{\varphi:[n]\to[m],\text{\rm injective}}\indicator{G^\varphi= H^\varphi}\\
	&\leq 1-(1-p)^{n\choose 2}\leq p{n\choose 2}
	\end{align*}
	Therefore, the claim is proved.
\end{proof}
\begin{theorem}\label{glimitvar}
	Let us define the metric $d$ on
	the space of graph limit as
	\[
	d(|G|,|H|):=\sum_{n\in\bbN}f(n)\sum_{F\in\scG_{n}}\abs{t(F;G)-t(F;H)}.
	\] 
	Then, if $f:\bbN\to\bbR$ satisfies $\sum_{n\in\bbN}f(n){n\choose 2}2^{n\choose 2}<\infty$, the total variation of $|\bfG|$ satisfying (i)--(iv) is finite: that is,
	\[
	\norm{|\bfG|}_{TV[0,1]}:=\limsup_{p\to0}\sum_{k=1}^{N_p}d(|\Gamma_{\tau^p_k}|,|\Gamma_{\tau^p_{k-1}}|)<\infty.
	\]
On the other hand, for any finite graph $F$, $t(F;\Gamma_t)$ has finite bounded variation. In other words, $|\bfG|$ has finite dimensional bounded variation on $[0,1]$.
\end{theorem}
\begin{proof}
	By Fubini theorem and dominated convergence theorem,
	\begin{align*}
	\limsup_{p\to0}\sum_{k=1}^{N_p}d(|\Gamma_{\tau^p_k}|,|\Gamma_{\tau^p_{k-1}}|)&=\limsup_{p\to0}\sum_{n\in\bbN}f(n)\sum_{F\in\scG_{n}}\sum_{k=1}^{N_p}\abs{t(F;\Gamma_{\tau^p_k})-t(F;\Gamma_{\tau^p_{k-1}})}\\
	&\leq \limsup_{p\to\infty}\sum_{n\in\bbN}f(n)\sum_{F\in\scG_{n}}{n\choose 2}pN_p\\
	&\leq \limsup_{p\to\infty}\sum_{n\in\bbN}f(n){n\choose 2}2^{n\choose 2}pN_p\leq C\sum_{n\in\bbN}f(n){n\choose 2}2^{n\choose 2}
	\end{align*}
	where $C=\limsup_{p\to0} pN_p<\infty$ is given by Corollary \ref{estN}.
If $f$ does not satisfy the integrable property, then we get finite dimensional locally bounded variation since
\begin{align*}
\lim_{p\to0}\sum_{k=1}^{N_p}\abs{t(F;\Gamma_{\tau^p_k})-t(F;\Gamma_{\tau^p_{k-1}})}\leq{n\choose 2}\limsup_{p\to0}pN_p<\infty
\end{align*}
\end{proof}
The following corollary is Theorem 3.3 of \cite{Crane:2015we}.
\begin{corollary}
	$|\bfG|$ has finite dimensional bounded variation on $[0,1]$ if we define the metric $d$ on
	the space of graph limit as
	\[
	d(|G|,|G'|):=\sum_{n\in\bbN}2^{-n}\sum_{F\in\scG_{n}}\abs{t(F;G)-t(F;G')}.
	\]
\end{corollary}
If we define $f$ so that it decreases fast enough, then we have \textit{infinite dimensional} bounded variation.
\begin{corollary}
	$|\bfG|$ has bounded variation on $[0,1]$ if we define the metric $d$ on
	the space of graph limit as
	\[
	d(|G|,|G'|):=\sum_{n\in\bbN}2^{-n^{2}}\sum_{F\in\scG_{n}}\abs{t(F;G)-t(F;G')}.
	\]
\end{corollary}

\bibliography{Untitled}{}

\begin{thebibliography}{1}

\bibitem{Crane:2015ef}
H.~Crane.
\newblock {Time-varying network models}.
\newblock {\em Bernoulli}, 21(3):1670--1696, 2015.

\bibitem{Crane:2015we}
H.~Crane.
\newblock {Dynamic random networks and their graph limits}.
\newblock {\em The Annals of Applied Probability}, 26(2):691--721, 2016.

\bibitem{Crane:2016vw}
H.~Crane.
\newblock {Exchangeable graph-valued Feller processes}.
\newblock {\em Probability Theory and Related Fields}, 168(3-4):849--899, 2016.

\bibitem{Eagleson:2008ix}
G.~K. Eagleson and N.~C. Weber.
\newblock {Limit theorems for weakly exchangeable arrays}.
\newblock {\em Mathematical Proceedings of the Cambridge Philosophical
  Society}, 84(01):123, 1978.

\end{thebibliography}
\bibliographystyle{abbrv}
\end{document}